\documentclass[12pt, psamsfonts]{amsart}
\usepackage{enumerate}
\usepackage{amsmath}
\usepackage{amsthm}
\usepackage{amssymb}
\usepackage{amscd}
\usepackage{amsfonts}
\usepackage{amsbsy}
\usepackage{epsfig}
\usepackage{ulem}
\usepackage[latin1]{inputenc}      
\usepackage{tikz}
\usetikzlibrary{decorations.pathmorphing}
\usetikzlibrary{decorations.pathreplacing}

\newtheorem{theorem}{Theorem}

\newtheorem{proposition}[theorem]{Proposition}
\newtheorem{corollary}[theorem]{Corollary}

\newtheorem*{claim*}{Claim}

\newfont\bbf{msbm10 at 12pt}

\def\eps{\varepsilon}

\def\RR{{\mathbb R}}
\def\C{{\mathbb C}}
\def\N{{\mathbb N}}

\def\B{{\mathcal B}}

\def\B{{\mathcal B}}

\def\le{\leqslant}
\def\ge{\geqslant}

\def\1{ {\hbox{{\it 1}} \!\! I} }

\begin{document}

\title[A Besicovitch-Morse function preserving the Lebesgue measure]
{A Besicovitch-Morse function preserving the Lebesgue measure}
\author{Jozef Bobok}

\author{Serge Troubetzkoy}

\address{Department of Mathematics of FCE\\Czech Technical University
in Prague\\
Th\'akurova 7, 166 29 Prague 6, Czech Republic}
\email{jozef.bobok@cvut.cz}

\address{Aix Marseille Univ, CNRS, Centrale Marseille, I2M, Marseille,
France\linebreak
postal address: I2M, Luminy, Case 907, F-13288 Marseille Cedex 9,
France}
\email{serge.troubetzkoy@univ-amu.fr}
\urladdr{www.i2m.univ-amu.fr/perso/serge.troubetzkoy/}  \date{}

\thanks{We thank the A*MIDEX project (ANR-11-IDEX-0001-02), funded
itself by the ``Investissements d'avenir'' program of the French
Government, managed by the French National Research Agency (ANR).
 The first  author was supported by the European Regional Development
 Fund, project No.~CZ 02.1.01/0.0/0.0/16\_019/0000778.}

\subjclass[2000]{37E05, 37B40, 46B25, 46.3}
\keywords{nowhere differentiable function, topological entropy}

\begin{abstract}
We continue the investigation of which non-differentiable maps can occur in
the framework of ergodic theory started in \cite{Bo91}.
We construct a Besicovitch-Morse function map which
preserves the Lebesgue measure.
We also show that the set of
Besicovitch functions is of first category in the set of continuous
functions which preserve the Lebesgue measure.
\end{abstract}
\maketitle
\vspace*{-0.8cm}
\section{Introduction}
In 1925 Besicovitch constructed a continuous function, $f: [0,1] \to
[0,1]$,  for which unilateral derivatives, finite
or infinite, do not
exists at any point \cite{Be25}.  A few years later, Pepper gave a
more geometric proof of the
same result \cite{Pe28}.
Saks has shown that such functions form a set of first category in the
space of all continuous
functions \cite{Sa32}. After this, Morse constructed a continuous
function with a stronger conclusion \cite{Mo38}, not only do
unilateral derivates not exist, but additionally
$$\max\{\vert D^+f(t)\vert,\vert D_+f(t)\vert\}=\max\{\vert
D^-f(t)\vert,\vert D_-f(t)\vert\}=\infty,~t\in [0,1].$$
See \cite{JP15} for a more detailed historical development.

We are interested in whether such non-differentiable maps can occur in
the framework of ergodic theory, more precisely  whether such nowhere
differentiable
functions can exist for a continuous map of $[0,1]$  which preserves
the Lebesgue measure.
Our main result is the existence of a Besicovitch-Morse function in
the space of continuous functions
preserving the Lebesgue measure (Theorem \ref{t:3}), improving
an earlier result of Bobok who showed the existence of a Besicovitch
function in this space \cite{Bo91}.
Furthermore, in analogy to Saks' classical theorem \cite{Sa32},  we
show that the set of
Besicovitch functions is of first category in the set of continuous
functions which preserve the
Lebesgue measure (Corollary \ref{cor:1}). Our construction of the
Besicovitch-Morse function is inspired by
Pepper's construction.

\section{Nowhere differentiable maps in $C(\lambda)$}

Let $I := [0,1]$.  Let  $\lambda$ denote the Lebesgue measure on $I$
and $\B$ the Borel sets in $I$.
Let $C(\lambda)$ consist of all continuous $\lambda$-preserving
functions from $I$ onto $I$, i.e.,
$$ C(\lambda)=\{f\colon~I\to I\colon~\forall
A\in\B,~\lambda(A)=\lambda(f^{-1}(A))\}. $$

We define the upper, lower, left and right \textit{Dini derivatives}
of $f$ at $t$:
\begin{align*}
D ^{+}f(t)& :=  \limsup_{\scriptstyle x\to t^{+} \atop \scriptstyle
x\in I}{\frac {f(x)-f(t)}{x-t}} \qquad
D _{+}f(t) :=  \liminf_{\scriptstyle x\to t^{+} \atop \scriptstyle
x\in I}{\frac {f(x)-f(t)}{x-t}}\\
D ^{-}f(t)& :=  \limsup_{\scriptstyle x\to t^{-} \atop \scriptstyle
x\in I}{\frac {f(x)-f(t)}{x-t}} \qquad
D _{-}f(t) :=  \liminf_{\scriptstyle x\to t^{-} \atop \scriptstyle
x\in I}{\frac {f(x)-f(t)}{x-t}}.
\end{align*}

We say that a finite one sided derivative exists at $t$ if $D^+f(t) =
D_+f(t) \in \RR$ or $D^-f(t) = D_-f(t) \in \RR$,
and that a finite or infinite one sided derivative exists at $t$ if
$D^+f(t) = D_+f(t) \in \RR \cup \{\pm \infty \}$ or $D^-f(t) = D_-f(t)
\in \RR \cup \{ \pm \infty\}$.
We introduce the following classes of continuous nowhere
differentiable functions

A {\it Besicovitch function} is an  $f\in\C(I,\RR)$ such that for
every $t\in I$, there is neither a finite or infinite right nor a
finite or infinite left derivative at $t$.

A  {\it Morse functions}, is an $f\in\C(I,\RR)$ such that
$$\max\{\vert D^+f(t)\vert,\vert D_+f(t)\vert\}=\max\{\vert
D^-f(t)\vert,\vert D_-f(t)\vert\}=\infty,~t\in I;
$$
we skip the left, resp.\ right  term of the $\max\{\}$ if $t$ is the
right, resp.\  left endpoint of the interval $I$.

We endow  $C(\lambda)$ with the uniform metric
$\rho (f,g) := \sup_{x \in I} |f(x) - g(x)|$.

\begin{proposition}\label{p:1}
$C(\lambda)$, endowed by the uniform metric $\rho$, is a complete
metric space. \end{proposition}
We leave the standard proof of this result to the reader.

Recall that  {\it a knot point} of function $f$ is a point $x$ where
$D^{+}f(x)=D^{-}f(x)=\infty$ and $D_{+}f(x)=D_{-}f(x)=-\infty$.
The following theorem states a consequence of more general result proved in \cite{Bo91}.

\begin{theorem}\label{t:1}~The $C(\lambda)$-typical function
has a knot point at $\lambda$-almost every point.
\end{theorem}
The next result generalizes a classical result  result of Saks
\cite{Sa32}.

\begin{corollary}\label{cor:1}
The set of Besicovitch functions is a meager set in  $C(\lambda)$.
\end{corollary}

\begin{proof}
We use the following well known result (see \cite[Theorem 7.3]{Sa37}):
\textit{if $D^+f(x) \ge 0$ for a.e. $x \in I$ and $D^+f(x) > - \infty$
for every $x \in I$, then $f$ is
non-decreasing.}

By Theorem \ref{t:1} there is a residual set $K\subset C(\lambda)$
such that each element
of $K$ has a knot point at $\lambda$ almost every point of $I$. Fix $f
\in K$,  we have
 $D^+f(x) = +\infty \ge 0$ a.e., and $f$ can not be non-decreasing.
Applying the above result, we conclude that $D^+(x_0) = - \infty$ for
at least one point $x_0 \in I$;
in particular $f$ is not a Besicovitch function.
\end{proof}

Now we state our main result.
\begin{theorem}\label{t:3}
There is a Besicovitch-Morse function in
 $C(\lambda)$.
\end{theorem}
\begin{proof} We begin by a sketch of our construction.  The first
step is to construct an irregular
Cantor staircase $f_{0} : [0,1/2] \to \RR$
then to extend by symmetry to a tent-like devils' staircase
map (see Figure \ref{fig:f0}).  Next we modify this map by replacing
each flat segment by
an affinely rescaled copy of $f_{0}$ pointing downwards, producing the
map $f_{1}$.
At each stage we will have a modify the resulting map by replacing the
flat segments by affinely rescaled
copies of the original map, the scaling becoming more skewed at each
step, and the direction alternates
between tent maps pointing up and down.

Given a $\sigma$ positive integer we construct a discontinuum $E_{\sigma} \subset
\left [0,\frac12\right ]$:
\begin{equation*}E_{\sigma}=\left [0,\frac12\right
]\setminus L_{\sigma},\ \
\text{where}\ \
L_{\sigma}=\bigcup_{m=1}^{\infty}\bigcup_{p=1}^{2^{m-1}}r_{m,p},\end{equation*}
the open intervals $r_{m,p}=(a_{m,p},b_{m,p})$ are chosen as
follows:\newline \noindent ($m=1$)\ $d_{1,1}=\left [0,\frac12\right
]$, $r_{1,1}\subset d_{1,1}$,
\begin{itemize}
\item[($K_{1,\sigma}$)] $b_{1,1}$ is the center of $d_{1,1}$,
    $\frac{\lambda(r_{1,1})}{\lambda(d_{1,1})}=\frac12-\frac1{2^{1+\sigma}}$;
\end{itemize}
\noindent ($m>1$, $m$ even), if $d_{m,1}\cdots d_{m,2^{m-1}}$ are (from
left to right) the intervals of the set $\left [0,\frac12\right
]\setminus
\bigcup_{q=1}^{m-1}\bigcup_{p=1}^{2^{q-1}}r_{q,p}$, then
$r_{m,p}\subset d_{m,p}$, and for a suitable increasing sequence
$(k_{\sigma}(m))_{m\ge 2,\text{ even}}$ of positive integers (to be
determined later)
\begin{itemize}
\item[($K^{\rm even}_m$)] $a_{m,p}=\min
    d_{m,p}+\frac{\lambda(d_{m,p})}{2^{k_{\sigma}(m)}}$,
    $b_{m,p}=\max
    d_{m,p}-\frac{\lambda(d_{m,p})}{2^{k_{\sigma}(m)}}$;
\end{itemize}

\noindent ($m>1$, $m$ odd), if $d_{m,1}\cdots d_{m,2^{m-1}}$ are (from
left to right) the intervals of the set $\left [0,\frac12\right
]\setminus
\bigcup_{q=1}^{m-1}\bigcup_{p=1}^{2^{q-1}}r_{q,p}$, then
$r_{m,p}\subset d_{m,p}$,
\begin{itemize}
\item[($K^{\rm odd}_{m,\sigma}$)] $b_{m,p}$ is the center of
    $d_{m,p}$ (we refer to this as the {\it center property}),
    $\frac{\lambda(r_{m,p})}{\lambda(d_{m,p})}=\frac1{2}-\frac1{2^{m+\sigma}}$.
\end{itemize}

Given a map $f: I \to I$ and $x,y \in I$, $x \ne y$, define
$$R(f,x,y) := \frac{f(x) - f(y)}{x-y}.$$

We consider  a continuous nondecreasing function $f_{0,\sigma}:\left
[0,1/2\right ]\rightarrow I$ satisfying $f_{0,\sigma}(0)=0$,
$f_{0,\sigma}(1/2)=1$, $f_{0,\sigma}$ constant on every interval
$r_{m,p}$ and satisfying for each $d_{m,p}:=[c,d]$, $b_{m,p}$
\begin{equation}\label{e:new}
R(f_{0,\sigma},b_{m,p},c) = R(f_{0,\sigma},d,c).
\end{equation}
Notice that this number is at least 2 for every $d_{m,p}$.
The function $f_{0,\sigma}$ is a Cantor steplike function.

Next let
$$k_{\sigma}(m) := 1 + \log_2 \left ( 1/2  +
\max_p\frac{\lambda(f_{0,\sigma}(d_{m,p}))}{\lambda(d_{m,p})} \right
);$$
this definition implies that for all even $m$ and $d_{m,p}=[c,d]$ we
have
\begin{equation}\label{e:14}
k_{\sigma}(m) > 2 \quad \text{   and  }  \quad 0<
R(f_{0,\sigma},d,a_{m,p})   \le  {1}/{2} .
\end{equation}

\begin{figure}[t]
\begin{tikzpicture}[scale=6]
\draw[opacity=0.3] (0,0) -- (0.5,1);
\draw [opacity=0.3] (0,0) -- (0.5,0) -- (0.5,1) -- (0,1) -- (0,0);
\draw[]  (89/192,23/24) -- (23/48,23/24);
\draw[]  (7/24,11/12) -- (11/24,11/12);
\draw[] (49/192,17/24) -- (13/48,17/24);
\draw[]  (1/8,1/2) -- (1/4,1/2);
\draw[] (65/576,17/36) -- (17/144,17/36);
\draw[] (1/72,4/9) -- (1/9,4/9);
\draw[] (1/576,4/18) -- (1/144,4/18);
\draw (0,0) -- (1/576,4/18);
\draw (1/144,4/18) -- (1/72,4/9);
\draw (1/9,4/9) -- (65/576,17/36);
\draw (17/144,17/36) --  (1/8,1/2) ;
\draw (1/4,1/2) -- (49/192,17/24);
\draw (13/48,17/24) -- (7/24,11/12);
\draw (11/24,11/12) -- (89/192,23/24);
\draw (23/48,23/24) -- (.5,1);
\draw[opacity=0.3] (1-0,0) -- (1-0.5,1);
\draw [opacity=0.3] (1-0,0) -- (1-0.5,0) -- (1-0.5,1) -- (1-0,1) --
(1-0,0);
\draw[]  (1-89/192,23/24) -- (1-23/48,23/24);
\draw[]  (1-7/24,11/12) -- (1-11/24,11/12);
\draw[] (1-49/192,17/24) -- (1-13/48,17/24);
\draw[]  (1-1/8,1/2) -- (1-1/4,1/2);
\draw[] (1-65/576,17/36) -- (1-17/144,17/36);
\draw[] (1-1/72,4/9) -- (1-1/9,4/9);
\draw[] (1-1/576,4/18) -- (1-1/144,4/18);
\draw (1-0,0) -- (1-1/576,4/18);
\draw (1-1/144,4/18) -- (1-1/72,4/9);
\draw (1-1/9,4/9) -- (1-65/576,17/36);
\draw (1-17/144,17/36) --  (1-1/8,1/2) ;
\draw (1-1/4,1/2) -- (1-49/192,17/24);
\draw (1-13/48,17/24) -- (1-7/24,11/12);
\draw (1-11/24,11/12) -- (1-89/192,23/24);
\draw (1-23/48,23/24) -- (1-.5,1);
\node at (0,-0.05) {\tiny $0$};
\node at (0.5,-0.05) {\tiny $1/2$};
\node at (1,-0.05) {\tiny $1$};
\end{tikzpicture}
\caption{The map $f_{0,\sigma}$}\label{fig:f0}
\end{figure}
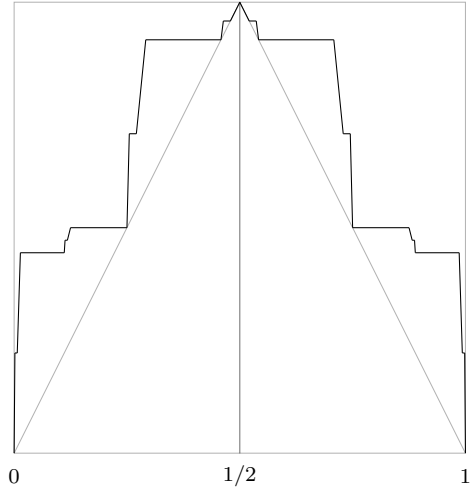

We extend $f_{0,\sigma}$ to the interval $I$ by  setting
 $$ f_{0,\sigma}(x):=f_{0,\sigma}(1-x),\ x\in \left [ 1/2,1\right ].
 $$  The function $f_{0,\sigma}$ and the interval $I$ form the
 \textit{basic $(\sigma=1)$-step triangle} of our construction.
The set $\{(x,f_{0,\sigma}(x));\ x\in \left [0,1/2\right ]\}$ is the
left side of the triangle, analogously the set
$\{(x,f_{0,\sigma}(x));\ x\in [1/2,1]\}$ is the right side
of the triangle. Now, we construct the desired function $f$ as
follows:

($c_0$) Start with the basic $(\sigma=1)$-step triangle with the base
$I$
and height $1$; the sides of the basic $(\sigma=1)$-step triangle
are the graph of  $f_{0}$ (see Figure \ref{fig:f0}). All
$E_1$-contiguous intervals, i.e., the holes
in the Cantor set $E_1$, and their counterparts in $\left [1/2,1\right
]$ will be called \textit{ $0$th $L$-segments} - the set of all $0$th
$L$-segments will be denoted by $\mathcal L_0$.

($c_1$) The flat segment corresponding to an interval $(a,b)\subset
r\in\mathcal L_0$ (and $f_{0}$) is the set
\begin{equation*}
\{(x,f_{0}(a))\colon~x\in [a,b]\};
\end{equation*}
  \noindent $m$ odd: for every element $r_{m,p}$ of $\mathcal L_0$
  construct affinely rescaled $(\sigma=1+m)$-step triangle whose base
  is the flat segment corresponding to $r_{m,p}$;

  \noindent $m$ even: for every element $r_{m,p}=(a,b)$ of $\mathcal
  L_0$ construct two affinely rescaled $(\sigma=1+m)$-step triangles
  whose bases are the flat segments corresponding to

 \begin{equation}\label{e:15}\left
 (a,b-\frac{1}{2^{k_1(m)}-2}(b-a)\right ),~\left
 (b-\frac{1}{2^{k_1(m)}-2}(b-a),b\right );\end{equation}
    constructed step  triangles are placed inwards the basic step triangle,
    the height of step triangle with the base corresponding to
    $r_{m,p}=(a,b)\subset d_{m,p}=[c,d]$ is equal to
\begin{equation}\label{e:9}f_0(a)-f_0(c).\end{equation}
The union of sides of all so far constructed step triangles defines the
function $f_1$. All new contiguous intervals (subintervals of some
previous $0$th $L$-segments) will be called
\textit{ $1$st $L$-segments} - the set of all $1$st $L$-segments will
be denoted by $\mathcal L_1$.\newline ($c_n$) Consider an element
$r^{n-1}_{m_{n-1},\cdot}=(a,b)$ from $\mathcal L_{n-1}$  satisfying
\begin{equation*}
r^{0}_{m_{0},\cdot}\supset r^{1}_{m_{1},\cdot}\supset \cdots \supset
r^{n-1}_{m_{n-1},\cdot},~r^{i}_{m_{i},\cdot}\in\mathcal L_i;
\end{equation*}
\noindent $m_{n-1}$ odd: construct affinely rescaled
$(\sigma=1+m_0+\cdots+m_{n-1})$-step triangle whose base is the flat
segment corresponding to $r^{n-1}_{m_{n-1},\cdot}$ and $f_{n-1}$;

\noindent $m_{n-1}$ even: construct two affinely rescaled
$(\sigma=1+m_0+\cdots+m_{n-1})$-step triangles whose bases are the flat
segments corresponding to $f_{n-1}$ and (\ref{e:15}) with
$k_{1+m_0+\cdots+m_{n-2}}(m_{n-1})$; the constructed step triangles are placed
inwards the bigger step triangle on whose side has its base, the height of
the step triangle corresponding to
$r^{n-1}_{m_{n-1},\cdot}=(a,b)\subset d^{n-1}_{m_{n-1},\cdot}=[c,d]$
is equal to $$\vert f_{n-1}(a)-f_{n-1}(c)\vert.$$ Realizing the
construction described above for all elements from $\mathcal L_{n-1}$,
the union of sides of all so far constructed step triangles define the
function $f_n$. All new contiguous intervals (subintervals of some
previous $(n-1)$st $L$-segments) will be called \textit{ $n$th
$L$-segments} - the set of all $n$th $L$-segments will be denoted by
$\mathcal L_n$.
\vglue2mm
Finally, put $f=\lim_{n\rightarrow \infty} f_n$ (obviously
$\rho(f_{n-1},f_n)\le\frac1{2^{n/2}}$). In what follows we will
repeatedly use the following easy consequence of our construction:

\begin{equation}\label{e:10}\forall n\in\N\cup\{0\}~\forall x\in
[0,1]\setminus\bigcup_{L\in\mathcal
L_n}\colon~f(x)=f_n(x).\end{equation}

In order to verify that the function $f$ is a Besicovitch-Morse
function we distinguish several cases.

\noindent {\bf I.} First, we assume that $x\in [0, 1]$ is not a point
of any $0$th $L$-segment. Because of symmetry we only consider points
from $[0,1/2]$.

\noindent {\bf I(+)} Assume that $x\in [0,1/2)$ is not the left
endpoint of any $0$th $L$-segment. We show that $f'_+(x)$ does not
exist and at least one of the right Dini derivatives of $f$ at $x$ is
infinite.

Fix $h>0$ arbitrarily small, then there is an odd $m$ such that for
some $p$, the $0$th $L$-segment
$r_{m,p} = (a,b)\subset d_{m,p}=[c,d]$  is  contained in
$(x,x+h)$.  We choose $p$ so that it is the left most such segment,
thus any $0$th $L$-segment $r_{m',p'}$
between $x$ and $a$ satisfies $m'>m$.  If $x < c$, then since $x$ is
not a point of any $0$th $L$-segment,
we would have another $L$-segment $r_{m',p'}$ between $x$ and $a$ with
$m' \le m$, thus
\begin{equation*}c\le x<a<\frac{a+b}{2}.\end{equation*}
Since by (\ref{e:10}) $f(t)=f_{0}(t)$ for $t\in\{c,x,a\}$, since $m$
is odd we obtain from (\ref{e:9})
\begin{equation}\label{e:2}f\left (\frac{a+b}{2}\right )=f(c).
\end{equation}
Furthermore, since $f_{0}$ is monotone on $E$, again using
\eqref{e:10}
\begin{equation}\label{e:2'}f(c)\le f(x)< f(a).
\end{equation}
The number $h$ was chosen arbitrarily small, so \eqref{e:2} and
\eqref{e:2'} imply
$$D_+f(x)\le 0\le D^+f(x).$$

Let us evaluate $\max\{\vert D^+f(x)\vert,\vert D_+f(x)\vert\}$.
Define $\eps_m\in (0,1]$ by
\begin{equation}\label{e:12}f(a)-f(x)=\eps_m[f(a)-f(c)]\end{equation}
Using the fact that $m$ is odd, and $b$ is in the middle of $[c,d]$,
($K^{\rm odd}_{m,1}$) implies
that $a-c = \lambda(d_{m,p})/2^{m+1}$, thus from \eqref{e:12}  we have
\begin{equation}\label{e:6}
\frac{f(a)-f(x)}{a-x}=\frac{\eps_m[f(a)-f(c)]}{a-x}\ge
\frac{\eps_m[f(a)-f(c)]}{a-c}
=\frac{\eps_m[f(a)-f(c)]}{\frac{\lambda(d_{m,p})}{2^{m+ 1}}}.
\end{equation}
Furthermore combining  \eqref{e:new}  with the fact that $b$ is in the
middle of $[c,d]$, and then the
fact that $R(f_{0},d,c)$ is at least 2, yields
\begin{equation}\label{e:6'}
\frac{\eps_m[f(a)-f(c)]}{\frac{\lambda(d_{m,p})}{2^{m+1}}}=\frac{\eps_m\frac{\lambda(f_{0}(d_{m,p}))}{2}}{\frac{\lambda(d_{m,p})}{2^{m+1}}}\ge
\frac{\eps_m\frac{2\lambda(d_{m,p})}{2}}{\frac{\lambda(d_{m,p})}{2^{m+1}}}=\eps_m2^{m+1}.
\end{equation}
So if $\eps_m\ge\eps$ for some positive $\eps$ and a sequence of odd
$m$'s, we immediately have $D^+f(x)=\infty$.

To the contrary assume that the sequence $(\eps_m)_{m\text{-odd}}$
converges to zero. In this case we will show
$D_+f(x) = - \infty$.
Let us denote $d^{(1)}_{1,1} := [c^{(1)},d^{(1)}] :=[a,\frac{a+b}{2}]$
and $r^{(1)}_{1,1}=(a^{(1)},b^{(1)})\in\mathcal L_1$,
$r^{(1)}_{1,1}\subset d^{(1)}_{1,1}$. Using
($c_1$) for the intervals $d^{(1)}_{1,1}, r^{(1)}_{1,1}$,  and the
center property we obtain

$$a^{(1)}-a=\frac{\lambda(d^{(1)}_{1,1})}{2}-\lambda(r^{(1)}_{1,1})=\frac{\lambda(d^{(1)}_{1,1})}{2^{1+m}}.
$$
Using  ($K^{\rm odd}_{m,1}$)  and the inequalities
$\lambda(d^{(1)}_{1,1})=\frac{\lambda(r_{m,p})}{2} <
\frac{\lambda(d_{m,p})}{2^2}$
(which both follow from the center property) we obtain
$$\frac{\lambda(d^{(1)}_{1,1})}{2^{1+m}} <
\frac{1}{2^2}\frac{\lambda(d_{m,p})}{2^{m+1}}=\frac{a-c}{2^2}.
$$
Thus
$$
a^{(1)}-x=a^{(1)}-a+a-x < \frac{a-c}{2^2}+a-c<2(a-c).
$$
By construction, the center property implies $f(a^{(1)}) = (f(a) +
f(c))/2$,
hence from (\ref{e:12}) we obtain
\begin{equation*}
- R(f,x,a^{(1)}) =\frac{(1/2-\eps_m)[f(a)-f(c)]}{a^{(1)}-x}> \left
({1/2}-\eps_m \right )  \frac{R(f,a,c)}{2}
\end{equation*}
But, \eqref{e:6} and \eqref{e:6'} imply $R(f,a,c) \ge 2^{m}$, and we
conclude
\begin{equation}\label{e:13}
- R(f,x,a^{(1)})   \ge -\left (1/2-\eps_m\right )\cdot 2^{m-1}.
 \end{equation}

By our assumption  $\eps_m$ converge to zero, so (\ref{e:13}) implies
$$-D_+f(x)=\infty=\max\{\vert D^+f(x)\vert,\vert D_+f(x)\vert\}.$$
We have already seen that $D_+f(x)\le 0\le D^+f(x)$, i.e., $f'_+(x)$
does not exist.

\noindent {\bf I(-)} Assume that $x\in (0,1/2]$ is not the right
endpoint of any $0$th $L$-segment. We show that $f'_-(x)$ does not
exist and $D^-f(x)=\infty$.

Fix $h>0$ arbitrarily small, let $r_{m,p} = (a,b)\subset
d_{m,p}=[c,d]$  be the $0$th $L$-segment contained in
$(x-h,x)$. W.l.o.g. we can assume that $m$ is even and that any $0$th
$L$-segment between $x$ and $b$ is labeled by an $m'>m$. Then
\begin{equation*}\label{e:51}b<x\le d\end{equation*}
and similarly as in  (\ref{e:2}) and \eqref{e:2'},
denoting by $e := e(m) :=  b - \frac{b-a}{2^{k_1(m)+1 } -
2^2}$ the middle of the interval $(b-\frac{1}{2^{k_1(m)}-2}(b-a),b)$,
from (\ref{e:15})
\begin{equation*}f(e) =f(c)<f(b)< f(x)\le f(d).
\end{equation*}
Using ($K^{\rm even}_m$), (\ref{e:14}) and (\ref{e:10}) we obtain for
each even $m$
\begin{align*}
&0<R(f,x,a) \le \frac{f(d)-f(a)}{x-a} = R(f,d,a) \frac{d-a}{x-a} < 1/2
\frac{d-a}{x-a}
\end{align*}
But for even $m$ we have
 $$\frac{d-a}{x-a} < \frac{d-a}{b-a} = \frac{2^{k_1(m)}-1}{2^{k_1(m)}
 -2},$$
 thus
\begin{equation}\label{e:16'}
0 < R(f,x,a)  <  1.
\end{equation}

Using $f(e)=f(c)$, $f(x) \ge f(b)$  and the definition of $k_1(m)$
yields the first inequality
\begin{align}\label{e:17}
& R(f,x,e) \ge
\frac{f(b)-f(c)}{2\frac{d-c}{2^{k_1(m)}}}=\frac{(2^{k_1(m)}-1)2^{k_1(m)
-1}}{2^{k_1(m)}}R(f,d,c) >2^{k_1(m)}-2;
\end{align}
while the equality follows from \eqref{e:new} and
 the last inequality follows from the fact that by our construction
 $\frac{\lambda(f(d_{m,p}))}{\lambda(d_{m,p})}\ge 2$ for each even $m$
 and each $p$.

When $h$ approaches $0$, the integer $m$ tends to $+ \infty$ and thus,
(\ref{e:16'}) and (\ref{e:17}) imply $D_-f(x)\le 1<D^-f(x)=\infty$.

\noindent {\bf II.} Second, we assume that for some
positive integer $n$, $x\in I$ is a point of some $n-1$st $L$-segment
and does not belong to any $n$th $L$ segment. Then the point
$(x,f(x))$ lies on the side of a
step triangle which is an affinely rescaled version of the basic
$(\sigma=1+m_0+\cdots+m_{n-1})$-step triangle; the facts that neither finite nor
infinite $f'_+(x),f'_-(x)$ exist and
$$\max\{\vert D^+f(x)\vert,\vert D_+f(x)\vert\}=\max\{\vert
D^-f(x)\vert,\vert D_-f(x)\vert\}=\infty$$
can be proven analogously as in {\bf I}.

\noindent {\bf III.}  Finally suppose that $x\in I$ belong
to $L$-segments of all orders, i.e.,
$\{x\}=\bigcap_{n=1}^{\infty}s_{m_n,p_n}$, where
$s_{m_n,p_n}=(a_n,b_n)$ equals to $r_{m_n,p_n}$ for $m_n$ odd, resp.\
$s_{m_n,p_n}\subset r_{m_n,p_n}$ for $m_n$ even and $r_{m_n,p_n}$
denotes the $(n-1)$st $L$-segment the point $x$ belongs to.
The function $f_{n}$ and the flat segment on the graph of $f_{n-1}$
corresponding to the interval
$(a_n,b_n)$ form a rescaled step triangle $\Delta_n$  which we use to
estimate the derivatives
$D^+f$, $D_+f$, $D^-f$, $D_-f$ at the point $x$.  From the
construction it follows that

\begin{enumerate}[i)]
\item $\Delta_n$ is oriented upwards for $n$ even, and downwards for
    $n$ odd and \\
$f((a_n,b_n))\supset f((a_{n+1},b_{n+1}))$ for each $n$.
\end{enumerate}

Denote $\ell_n$, resp.\  $h_n$ the  length of the base, resp.\  height
of $\Delta_n$. In our construction at each step on the flat segment of
$f_{n-1}$ we build a ``tent'' consisting of two sides of $\Delta_n$
in such a way that
\begin{equation*}R(f,a_n,\gamma_n)>2
R(f,a_{n-1},\gamma_{n-1})\text{ for each }n,
\end{equation*}
(here $\gamma_n :=  a_n + (b_n - a_n)/2$), hence
\begin{enumerate}[ii)]
\item $\lim_{n\to\infty}h_n/\ell_n=\infty$.
\end{enumerate}

It follows from i)  that  $D^+(x) \ge 0 \ge D_+(x)$ and $D^-(x) \ge 0
\ge D_-(x)$.
 Thus if  $f'_+(x)$ exists then it equals 0,  but this is impossible
 if $f \in C(\lambda)$
(and similarly for $f'_-(x)$).  We will prove that $f \in C(\lambda)$
below, and thus
$f$ can not be a Besicovitch function.

The set $\{(t,f_n(t));\ t\in (a_n,\gamma_n)\}$ is the left side of
$\Delta_n$. By symmetry, we can suppose without loss of generality,
that the point $x$ corresponds to the left side of the step triangles for
infinitely many $n$.

\begin{figure}[h]
\begin{minipage}[ht]{0.49\linewidth}
\centering
\begin{tikzpicture}[scale=0.6]

\draw (0,0) -- (2.5,10) -- (5,0) -- (0,0);
\draw (0,-0.8) -- (5,-0.8);
\draw[dotted]  (0,0) -- (0,10)-- (2.5,10) -- (2.5,0);
\draw[dashed] (0,5) -- (2.5,5);
\node at (0,-1.2) {\tiny $a_n$};
\node at (2.5,-1.2) {\tiny $\gamma_n$};
\node at (5,-1.2) {\tiny $b_n$};

\node at (0.9,5.5) {\tiny $A$};
\node at (1.8,5.5) {\tiny $B$};
\node at (1.6,4.5) {\tiny $C$};
\node at (0.7,4.5) {\tiny $D$};
\end{tikzpicture}
\end{minipage} \nobreak
\begin{minipage}[ht]{0.49\linewidth}
\centering
\begin{tikzpicture}[scale=0.6]

\draw[] (0,10) -- (2.5,0) -- (5,10) -- (0,10);
\draw (0,-0.8) -- (5,-0.8);
\draw[dotted]  (2.5,10) -- (2.5,0) -- (0,0) -- (0,10);
\draw[dashed] (0,5) -- (2.5,5);
\node at (0,-1.2) {\tiny $a_n$};
\node at (2.5,-1.2) {\tiny $\gamma_n$};
\node at (5,-1.2) {\tiny $b_n$};

\node at (0.9,4.5) {\tiny $A$};
\node at (1.8,4.5) {\tiny $B$};
\node at (1.6,5.5) {\tiny $C$};
\node at (0.7,5.5) {\tiny $D$};
\end{tikzpicture}
\end{minipage}
\caption{$\Delta_n$ oriented upwards (left) and downwards
(right)}\label{fig:serge}
\end{figure}
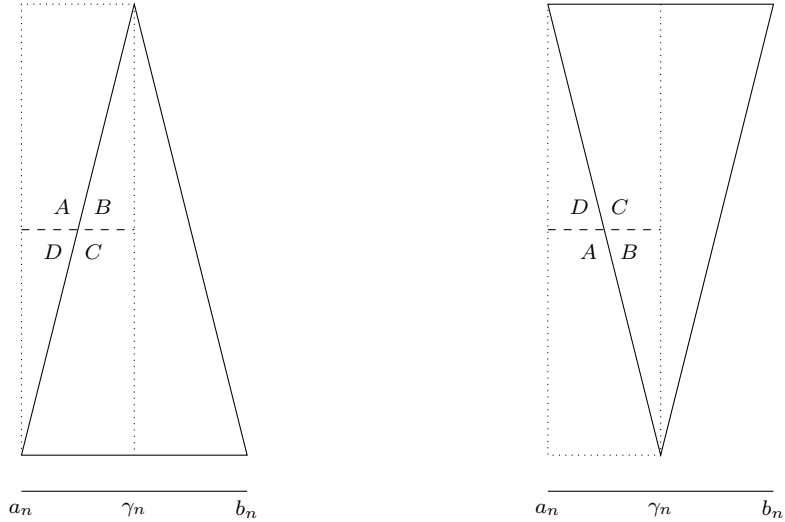

Consider Figure \ref{fig:serge}, the horizontal dotted line cuts the
step triangle $\Delta_n$ in the middle (in height).
The point $(x,f(x))$ must be in one of the sets $A,B,C$ or $D$. For
convenience we take
$A,B$ open, and $C,D$ closed.

{\bf III$_1$.} From ii) we deduce that  if for infinitely
many $n$
\begin{itemize}
\item[(1)]  the step  triangle $\Delta_n$ is oriented upwards, resp.
    downwards and $(x,f(x))\in A\cup B$, then $D_+f(x)  = - \infty$
    and $D^-f(x)  = \infty$, resp. $D^+f(x)  = \infty$ and $D_-f(x)
    = -\infty$,
\item[(2)] the step  triangle $\Delta_n$ is oriented upwards, resp.
    downwards and $(x,f(x))\in C\cup D$, then $D^+f(x)  = \infty$,
    resp. $D_+f(x)  = -\infty$.
\end{itemize}

{\bf III$_2$.} The argument for $D^-f(x)$ and $D_-f(x)$ when
$(x,f(x))\in C \cup D$ is more complicated.  First of all notice that
the above argument works without
change if we replace the dashed line
in the middle of the figure with a line at any fixed percentage of the
height. Thus the remain case
is when this percentage tends to zero.

Denote the percentage of the height of $\Delta_n$ corresponding to the
position of $(x,f(x))$ by
$\alpha_n= \frac{\vert f(x)-f(a_n)\vert}{h_n}\in (0,1)$.
Our assumption is
$$\lim_{n\to\infty}\alpha_n=0.$$
Let $n$ be even and sufficiently large to satisfy $\alpha_m\in
(0,1/2)$ for each $m\ge n$.
By our construction and the definition of $\alpha_n,\alpha_{n+1}$ (see
Figure \ref{fig:serge'})
 \begin{equation}\label{e:20}f(a_{n+1})-\alpha_{n+1}h_{n+1}=f(a_n)+\alpha_nh_n,~\alpha_nh_n\ge
 (1-\alpha_{n+1})h_{n+1}.
 \end{equation}

\begin{figure}[h]
\centering
\begin{tikzpicture}[scale=0.6]
\draw (0,0) -- (2.5,10) -- (2.5,0) -- (0,0);;
\draw[dashed] (0,1) -- (2.5,1);
\draw[] (0.1,1.5) -- (0.6,1.5) -- (0.25,0.3) -- (0.1,1.5);
\draw[dashed] (0.65,1.5) -- (2.5,1.5);
\draw
[decorate,decoration={brace,amplitude=2pt,mirror,raise=4pt},yshift=0pt]
(2.5,1.05) -- (2.5,1.5) node [black,midway,xshift=0.8cm] {\tiny
$ \qquad  \alpha_{n+1}h_{n+1}$};
\draw
[decorate,decoration={brace,amplitude=1.5pt,mirror,raise=4pt},yshift=0pt]
(0,0.3) -- (0,0) node [black,midway,xshift=0.8cm]
{\tiny \hspace{-5.5cm} $\alpha_nh_n - (1-\alpha_{n+1} )h_{n+1}$};

\draw
[decorate,decoration={brace,amplitude=2pt,mirror,raise=4pt},yshift=0pt]
(2.5,0) -- (2.5,0.95) node [black,midway,xshift=0.8cm]
{\tiny $\quad \alpha_nh_n$};

\draw (0,-0.8) -- (2.5,-0.8);
\node at (0,-1.2) {\tiny $a_n$};
\node at (2.5,-1.2) {\tiny $\gamma_n$};

\end{tikzpicture}
\caption{The upper left endpoint of the small triangle has coordinates
$(a_{n+1}, f(a_{n+1}))$.}\label{fig:serge'}
\end{figure}
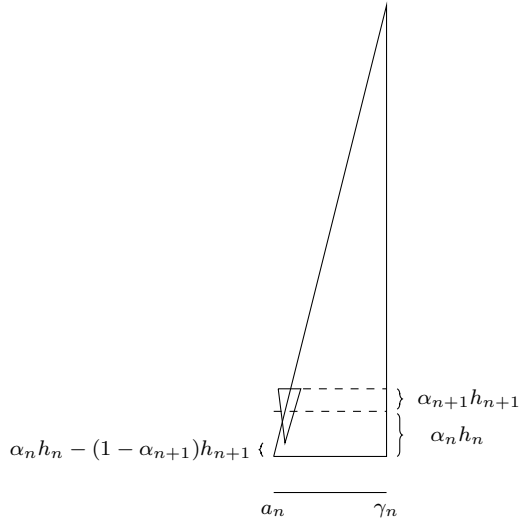

Assume that
\begin{equation}\label{e:19}\nonumber
f(a_{n+1})> f(a_n)+2\alpha_nh_n;
\end{equation}
using (\ref{e:20}) we get
 \begin{equation}\label{e:24}f(a_{n+1})-f(a_n)=\alpha_{n+1}h_{n+1}+\alpha_nh_n>
 2\alpha_nh_n
 \end{equation}
hence again from (\ref{e:20})
$$\alpha_{n+1}h_{n+1}>\alpha_nh_n\ge (1-\alpha_{n+1})h_{n+1}
$$
and $\alpha_{n+1}\ge 1/2$, a contradiction with our choice of
$\alpha_{n+1}$. It shows that
\begin{equation*}f(a_{n+1})\le
f(a_n)+2\alpha_nh_n.\end{equation*}

Choose $\kappa(n) \in \N$ such that  $\alpha_n\in
[1/2^{\kappa(n)+2},1/2^{\kappa(n)+1})$; this implies
\begin{equation}\label{e:21}
f(a_{n+1})\le  f(a_n)+\frac{h_n}{2^{\kappa(n)}}.
\end{equation}

In our basic construction for each fixed $m$ there are
finitely many $0$th $L$-segments
$r^{(0)}_{m,p}$ and (see Figure~\ref{fig:f0})
\begin{equation*}
f_0(0)+\frac{1}{2^m}=\frac{1}{2^m}\le f_0(r^{(0)}_{m,1})<f_0(r^{(0)}_{m,2})<\cdots<f_0(r^{(0)}_{m,2^{m-1}}),
\end{equation*}
hence analogously for each $n$, for $f_n$ on $(a_n,b_n)$ and for each fixed $m$ we have
\begin{equation}\label{e:1}
f_n(a_n)+\frac{h_n}{2^m}\le f_n(r^{(n)}_{m,1})<f_n(r^{(n)}_{m,2})<\cdots<f_n(r^{(n)}_{m,2^{m-1}}),
\end{equation}
where $r^{(n)}_{m,p}$ are $n$th $L$-segments corresponding to the flat segments on the left side of $\Delta_n$.

Since $f(a_n)=f_n(a_n)$ and $f(a_{n+1})=f_n(a_{n+1})=f_n(s_{m_{n+1},p_{n+1}})\subset f_n(r^{(n)}_{m_{n+1},p_{n+1}})$, choosing $m = m_{n+1}$  in (\ref{e:1}) and combining with
 (\ref{e:21}) yields
\begin{equation}\label{e:hi}
m_{n+1}\ge \kappa(n).
\end{equation}
This enables us to estimate the length $\ell_{n+1}$. By our
construction, for each $m=2j$ or $m=2j+1$,
the leftmost segment $r^{(0)}_{m,1}$
of the $0$th category satisfies
\begin{align}\label{a:1}&\lambda(r^{(0)}_{2j+1,1})\le
\frac{1}{2^{2+j^2 + j +\sum_{i=1}^{j}k(2i)}}\le \frac{1}{2^{2j+1}},
\quad j \ge 0,\\
&\label{a:2}\lambda(r^{(0)}_{2j,1})\le
\frac{1}{2^{1+j^2+j+\sum_{i=1}^{j-1}k(2i)}}\le \frac{1}{2^{2j}}, \quad
j \ge 1.
\end{align}
Moreover, all segments from $\mathcal L_0$ placed to the left of
$r^{(0)}_{m,1}$ are shorter than $r^{(0)}_{m,1}$.
Since $\Delta_n$ and $\Delta_{n+1}$ are affinely rescaled version of
the basic  $\sigma$-step triangle with large $\sigma$, we
conclude from \eqref{e:hi},  (\ref{a:1}) and (\ref{a:2})  that
\begin{equation}\label{e:22}
\lambda(s_{m_{n+1},p_{n+1}})=\ell_{n+1}\le
\lambda(r^{(n)}_{\kappa(n),1})\le \frac{\ell_n}{2^{\kappa(n)}}.
\end{equation}
With the help of (\ref{e:24}) we can write
\begin{align}\label{a:3}
R(f,x,a_n)&\ge\frac{f(a_{n+1})-\alpha_{n+1}h_{n+1}-f(a_n)}{\frac{a_{n+1}+b_{n+1}}{2}-a_n}=\frac{\alpha_{n}h_{n}}{\frac{b_{n+1}-a_{n+1}}{2}+a_{n+1}-a_n}=\\
\nonumber
&=\frac{\alpha_{n}h_{n}}{\frac{\ell_{n+1}}{2}+a_{n+1}-a_n}=\frac{h_n}{\ell_n}\cdot\frac{\ell_n}{\ell_{n+1}}\cdot\frac{\alpha_n}{\frac{1}{2}+\frac{a_{n+1}-a_n}{\ell_{n+1}}}.
\end{align}

{\bf III$^{\textbf{b}}_2$.} If $\frac{a_{n+1}-a_n}{\ell_{n+1}}$ is
bounded, since $a_{n+1} > a_n$ we conclude that
$\frac{1}{\frac{1}{2}+\frac{a_{n+1}-a_n}{\ell_{n+1}}}$ is positive and
bounded by a constant $C$.
Thus using \eqref{e:22} and \eqref{a:3}, ii) and the definition of
$\kappa(n)$ yields
$$
R(f,x,a_n)\ge \frac{h_n}{\ell_n} 2^{\kappa(n)} {\alpha_n}{C} \ge C
\frac{h_n}{4\ell_n} \to \infty.
$$
We conclude that $D^-f(x) = + \infty$ and finishes the proof of the
fact
that $f$ is Besicovitch-Morse when this ratio is bounded.

{\bf III$^{\textbf{u}}_2$.} Finally assume that the ratio
$\frac{a_{n+1}-a_n}{\ell_{n+1}}$ is not bounded.
If the liminf of these ratios  over $n$ even is finite, we can use the
corresponding subsequence
and conclude in the same way.
Thus we can assume that the limit over even $n$ is infinite.
Then
\begin{align*}
R(f,x,a_n)  &= \frac{f(x) - f(a_{n+1}) + f(a_{n+1}) - f(a_n)}{x -
a_{n+1} + a_{n+1} - a_n}\\ \nonumber
& = R(f,a_{n+1}, a_n)  \frac{1 + \frac{ f(x) - f(a_{n+1})}{f(a_{n+1})
- f(a_n)}}{1+ \frac{x - a_{n+1}}{a_{n+1} - a_n}}.
\end{align*}

But $\ell_{n+1}/2 \ge x  - a_{n+1} >0$, so
$$
R(f,x,a_n)  \ge R(f,a_{n+1}, a_n)  \frac{1 + \frac{ f(x) -
f(a_{n+1})}{f(a_{n+1}) - f(a_n)}}{1+ \frac{\ell_{n+1}}{2(a_{n+1} -
a_n)}}.
$$

From the definiton of $\alpha_n$ we have
$$0 <
\frac{  f(a_{n+1}) - f(x) }{f(a_{n+1}) - f(a_n)} \le
\frac{  f(a_{n+1}) - f(x) }{h_{n+1}} = \alpha_{n+1},$$  by our
assumption on the unboundedness of the ratios we have
$$\frac{\ell_{n+1}}{2(a_{n+1} - a_n)} \to 0,$$
and by our construction
$$R(f,a_{n+1}, a_n) \ge \frac{2h_n}{\ell_n}.$$
Combining the last four equations we conclude
$D^-f(x) = \infty,$
and our proof that $f$ is Besicovitch-Morse is finished.

In order to finish our proof let us show that the function $f$
preserves the Lebesgue measure. To this end let us define a new
sequence $(g_n)_{n\ge 0}$ of functions from $C(\lambda)$ for which
$\lim_{n\to\infty}g_n=f$.

We define $g_0$ as the full tent map, i.e., the function
$g_0(x)=1-\vert 1-2x\vert$, $x\in I$.
To define the function $g_n$ we put
$$
g_n:=f_n\text{ on }I\setminus\bigcup\mathcal L^*_{n-1},
$$
where $\mathcal L^*_{n-1}$ denotes the set of all $(n-1)$st
$L$-segments and their counterparts in $[1/2,1]$. On each element of
$\mathcal L^*_{n-1}$ instead of rescaled step triangle we use a
rescaled tent map of the same base, height and orientation. Then
$$\lim_{n}g_n=\lim_nf_n=f,
$$
so it is sufficient to show that each $g_n\in C(\lambda)$.
It is true for $g_0$. Let $g_1^{(0)} := g_0$ and using the
lexicographical order on $(m,p)$ (first $m$, then $p$) we
consider the $j$th-interval $r_{m,p}$ and its counterpart in $[1/2,1]$ to modify
$g_1^{(j-1)}$ to a map $g_1^{(j)}$ as in the sequence of pictures.
Property \eqref{e:new} implies that each of these modifications is in
$C(\lambda)$, then Proposition \ref{p:1} implies $g_{1} := \lim_{j\to\infty}g_1^{(j)}\in C(\lambda)$. In order to verify that $g_n\in C(\lambda)$ we put $g_n^{(0)}=g_{n-1}$ and define the sequence $g_{n}^{(j)}$, $j\ge 1$, in an analogous way to (Figure \ref{fig:serge''}) on each element of $\mathcal L^*_{n-1}$.
\end{proof}

\begin{figure}[t]
\begin{minipage}[ht]{0.33\linewidth}
\centering
\begin{tikzpicture}[scale=3]
\draw (0,0) -- (1,0) -- (0.5,1) -- (0,0);
\draw[opacity=0.3] (0,0) -- (0.5,1);
\draw[opacity=0.5]  (89/192,23/24) -- (23/48,23/24);
\draw[opacity=0.5]  (7/24,11/12) -- (11/24,11/12);
\draw[opacity=0.5] (49/192,17/24) -- (13/48,17/24);
\draw[opacity=0.5]  (1/8,1/2) -- (1/4,1/2);
\draw[opacity=0.5] (65/576,17/36) -- (17/144,17/36);
\draw[opacity=0.5] (1/72,4/9) -- (1/9,4/9);
\draw[opacity=0.5] (1/576,4/18) -- (1/144,4/18);
\draw [opacity=0.5] (0,0) -- (1/576,4/18);
\draw [opacity=0.5] (1/144,4/18) -- (1/72,4/9);
\draw [opacity=0.5] (1/9,4/9) -- (65/576,17/36);
\draw [opacity=0.5] (17/144,17/36) --  (1/8,1/2) ;
\draw [opacity=0.5] (1/4,1/2) -- (49/192,17/24);
\draw [opacity=0.5] (13/48,17/24) -- (7/24,11/12);
\draw [opacity=0.5](11/24,11/12) -- (89/192,23/24);
\draw [opacity=0.5] (23/48,23/24) -- (.5,1);
\draw[opacity=0.3] (1-0,0) -- (1-0.5,1);
\draw[opacity=0.5]  (1-89/192,23/24) -- (1-23/48,23/24);
\draw[opacity=0.5]  (1-7/24,11/12) -- (1-11/24,11/12);
\draw[opacity=0.5] (1-49/192,17/24) -- (1-13/48,17/24);
\draw[opacity=0.5]  (1-1/8,1/2) -- (1-1/4,1/2);
\draw[opacity=0.5] (1-65/576,17/36) -- (1-17/144,17/36);
\draw[opacity=0.5] (1-1/72,4/9) -- (1-1/9,4/9);
\draw[opacity=0.5] (1-1/576,4/18) -- (1-1/144,4/18);
\draw [opacity=0.5] (1-0,0) -- (1-1/576,4/18);
\draw [opacity=0.5] (1-1/144,4/18) -- (1-1/72,4/9);
\draw [opacity=0.5] (1-1/9,4/9) -- (1-65/576,17/36);
\draw[opacity=0.5] (1-17/144,17/36) --  (1-1/8,1/2) ;
\draw [opacity=0.5] (1-1/4,1/2) -- (1-49/192,17/24);
\draw [opacity=0.5] (1-13/48,17/24) -- (1-7/24,11/12);
\draw [opacity=0.5] (1-11/24,11/12) -- (1-89/192,23/24);
\draw [opacity=0.5] (1-23/48,23/24) -- (1-.5,1);
\node at (0,-0.05) {\tiny $0$};
\node at (0.5,-0.05) {\tiny $1/2$};
\node at (1,-0.05) {\tiny $1$};
\end{tikzpicture}
\end{minipage}\nolinebreak
\begin{minipage}[ht]{0.33\linewidth}
\centering
\begin{tikzpicture}[scale=3]
\draw (0,0) -- (1,0);
\draw (3/4,1/2) --  (0.5,1) -- (1/4,1/2);
\draw[]  (0,0) -- (1/8,.495) -- (3/16,0.005) -- (1/4,1/2);
\draw[]  (1-0,0) -- (1-1/8,.495) -- (1-3/16,0.005) -- (1-1/4,1/2);
\draw[opacity=0.2]  (1/8,1/2) -- (1/4,1/2);
\draw[opacity=0.2]  (1-1/8,1/2) -- (1-1/4,1/2);
\node at (0,-0.05) {\tiny $0$};
\node at (0.5,-0.05) {\tiny $1/2$};
\node at (1,-0.05) {\tiny $1$};
\end{tikzpicture}
\end{minipage}\nolinebreak
\begin{minipage}[ht]{0.33\linewidth}
\centering
\begin{tikzpicture}[scale=3]
\draw (0,0) -- (1,0);
\draw (3/4,1/2) --  (0.5,1) -- (1/4,1/2);
\draw[]  (1/9,4/9) -- (1/8,1/2) -- (3/16,0) -- (1/4,1/2);
\draw[]  (1-1/9,4/9) -- (1-1/8,1/2) -- (1-3/16,0) -- (1-1/4,1/2);
\draw[] (0,0) --(1/72,4/9)-- (4/72,0) -- (7/72,4/9) -- (15/144,0)  --
(1/9,4/9);
\draw[] (1-0,0) --(1-1/72,4/9)-- (1-4/72,0) -- (1-7/72,4/9) --
(1-15/144,0)  -- (1-1/9,4/9);

\draw[opacity=0.2] (1/72,4/9) -- (1/9,4/9);
\draw[opacity=0.2] (1-1/72,4/9) -- (1-1/9,4/9);
\node at (0,-0.05) {\tiny $0$};
\node at (0.5,-0.05) {\tiny $1/2$};
\node at (1,-0.05) {\tiny $1$};
\end{tikzpicture}
\end{minipage}

\caption{i) $g_1^{(0)}$ and $f_{0}$, ii) $g_1^{(1)}$ (the segment
$r_{1,1}$ is drawn for comparison), iii) $g_1^{(2)}$  (the segment
$r_{2,1}$ is drawn for comparison)}\label{fig:serge''}
\end{figure}
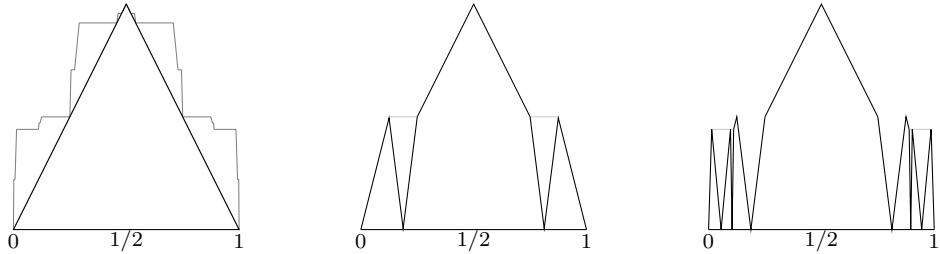

\end{document}